\documentclass{amsart}
\usepackage{graphicx}
\usepackage{amsmath,amssymb}
\usepackage{amsthm}
\usepackage{color}

\newtheorem{thm}{Theorem}[section]
\newtheorem{cor}[thm]{Corollary}
\newtheorem{lem}[thm]{Lemma}
\newtheorem{prop}[thm]{Proposition}

\theoremstyle{definition}
\newtheorem{defn}[thm]{Definition}
\theoremstyle{remark}
\newtheorem{rem}[thm]{Remark}
\theoremstyle{example}

\theoremstyle{conjecture}

\numberwithin{equation}{section}

\newcommand{\B}{{\mathbb B}}

\newcommand{\R}{{\mathbb R}}
\newcommand{\SSS}{{\mathbb S}}
\newcommand{\C}{{\mathbb C}}

\newcommand{\N}{{\mathbb N}}

\newcommand{\Z}{{\mathbb Z}}

\newcommand{\im}{{\rm Im}\,}

\newcommand{\calB}{{\mathcal B}}

\newcommand{\calD}{{\mathcal D}}

\newcommand{\calM}{{\mathcal M }}

\newcommand{\calT}{{\mathcal T}}

\newcommand{\calW}{{\mathcal W}}

\newcommand{\rpm}{\raisebox{.2ex}{$\scriptstyle\pm$}}

\def\one{\mbox{1\hspace{-4.25pt}\fontsize{12}{14.4}\selectfont\textrm{1}}}

\makeatletter
\newcommand{\doublewidetilde}[1]{{%
  \mathpalette\double@widetilde{#1}%
}}
\newcommand{\double@widetilde}[2]{%
  \sbox\z@{$\m@th#1\widetilde{#2}$}%
  \ht\z@=.9\ht\z@
  \widetilde{\box\z@}%
}
\makeatother

\begin{document}

\title[Sparse Domination of Composition Operators]{Sparse domination of weighted composition operators on weighted Bergman spaces}%

\author{Bingyang Hu, Songxiao Li$^\dag$, Yecheng Shi, and Brett D. Wick}

\date{\today}%

\address{Bingyang Hu: Department of Mathematics, University of Wisconsin, Madison, WI 53706-1388, USA.}%
\email{bhu32@wisc.edu}%%

\address{Songxiao Li:  Institute of Fundamental and Frontier Sciences, University of Electronic Science and Technology of China, 610054, Chengdu, Sichuan, P.R. China.}
\email{jyulsx@163.com}%%

\address{Yecheng Shi:  School of Mathematics and Statistics, Lingnan Normal University, Zhanjiang 524048, Guangdong P.R. China.}
\email{09ycshi@sina.cn}%%

\address{Brett D. Wick: Department of Mathematics and Statistics, Washington University in Saint Louis, Saint Louis, MO 63130-4899, USA.}
\email{wick@math.wustl.edu}



\thanks{$\dag$ Corresponding author.}

% ----------------------------------------------------------------

\maketitle

% ----------------------------------------------------------------

\begin{abstract}
The purpose of this paper is to study sparse domination estimates of composition operators in the setting of complex function theory. The method originates from proofs of the $A_2$ theorem for Calder\'on-Zygmund operators in harmonic analysis. Using this tool from harmonic analysis,  some new characterizations are given for  the boundedness and compactness of weighted composition operators acting between weighted Bergman spaces in the upper half plane. Moreover, we establish a new weighted type estimate for the holomorphic Bergman-class functions, for a new class of weights, which is adapted to Sawyer--testing conditions. We also extend our results to the unit ball $\B$ in $\C^n$.
\end{abstract}

% ----------------------------------------------------------------
\bigskip

\section{Introduction}

Let $\R^2_+:=\left\{z \in \C, \im z>0\right\}$ be the upper half-plane on the complex plane, $\widehat{\R^2_+}:=\R^2_+ \cup \{\infty\}$ and $H(\R^2_+)$ be the set of all holomorphic functions in $\R^2_+$ with the usual compact open topology. For $0<p<\infty$ and $\alpha>-1$, let $L^p_\alpha=L^p_\alpha(\R^2_+)$ be the collection of measurable functions $f$ in $\R^2_+$, for which the (quasi-) norm
\begin{equation} \label{norm01}
\|f\|_{p, \alpha}:= \left( \int_{\R^2_+} |f(z)|^pdA_\alpha(z) \right)^{\frac{1}{p}}
\end{equation}
is finite, where $dA_\alpha(z)=\frac{1}{\pi} (\alpha+1) (2\im z)^\alpha dA(z)$, $dA(z)=dxdy$, and $z=x+iy$.

The \emph{weighted Bergman space} $A_\alpha^p$ on $\R^2_+$ is defined to be the space $L_\alpha^p \cap H(\R^2_+)$. It is well known that when $1 \le p<\infty$, $A_\alpha^p$ is  a Banach space with the norm \eqref{norm01}; while for $p \in (0, 1)$, it is a Fr\'echlet space with the translation invariant metric
$$
d(f, g):=\|f-g\|^p_{p, \alpha}, \quad f, g \in A_\alpha^p.
$$
    We refer the interested reader to the books \cite{HKZ, Zhu} for more information about weighted Bergman spaces on the  unit disk and the unit ball.

Let $u \in H(\R^2_+)$ and $\varphi: \R^2_+ \rightarrow \R^2_+$ be a holomorphic  self-mapping. The \emph{weighted composition operator} is defined as
$$
W_{u, \varphi}(f)(z)=u(z) \cdot f \circ \varphi(z), \quad f \in H(\R^2_+), z \in \R^2_+.
$$
If $u(z) \equiv 1$, then $W_{u, \varphi}$ becomes the \emph{composition operator} and is denoted by $C_\varphi$, and if $\varphi(z)=z$, then $W_{u, \varphi}$ becomes the \emph{multiplication operator} and is denoted by $M_u$. See, for example, \cite{cz1, LS,SM} for more information about composition operators and weighted composition operators on weighted Bergman spaces on the unit disk.

In the recent decade, the sparse domination technique was developed and studied by many mathematicians  working in harmonic analysis. This technique dates back to Andrei Lerner from his alternative, simple proof of the $A_2$ theorem \cite{AL1, AL2}, proved originally by Hyt\"onen \cite{TH}. In Lerner's work, he was able to bound all Calder\'on-Zygmund operators by a supremum of a special collection of dyadic, positive operators called \emph{sparse operators}.  This estimate led almost instantly to a proof of the sharp dependence of the constant in related weighted norm inequalities, the $A_2$ theorem, a problem that had been actively worked on for over a decade.

Later, there have been many improvements to Lerner's techniques, as well as extending his ideas to a wide range of spaces and operators, such as \cite{CR, CPO, LM, AN}. In general, sparse bounds have been recognized as a finer quantification of the boundedness of an operator, which roughly says that the behavior of an operator can be captured by a ``sparse" collection of dyadic cubes.

Sparse bounds of operators acting between complex function spaces is a recent research topic. As far as we know, this type of estimates first appears in the work of Aleman, Pott and Reguera \cite{APR}, where they proved a pointwise sparse domination estimate of the Bergman projection to study the Sarason conjecture on the Bergman spaces.  Later, by using similar ideas, Rahm, Tchoundja and the last author \cite{RTW} were able to establish some weighted estimates for the Berezin transforms and Bergman projections acting between weighted Bergman spaces on the unit ball (see, also \cite{HW} for its analog in Hartogs domains).

The aim of this paper is to study the sparse domination estimate of the weighted composition operators acting on complex function spaces. The novelty are twofold.

\begin{enumerate}
\item [$(a)$ ] From the viewpoint of harmonic analysis, the weighted composition operators that we study, lack an integral structure and aren't immediately amenable to study via a dyadic structure. This is very different from the case of studying sparse bounds of the Hardy-Littlewood maximal operators, Calder\'on-Zygumund operators, Haar shift operators or other operators that have been considered in harmonic analysis. We will overcome this difficulty by applying integral representations of holomorphic Bergman-class functions and introduce some proper positive sparse forms which are adapted to the Carleson measure induced by  weighted composition operators (see, \eqref{010101}). Moreover, we are also able to describe the  compactness of weighted composition operators by using sparse domination. To the best of our knowledge, no prior results on describing the compactness of operators by using sparse domination exist in the literature. \\

\item [$(b)$] From the view of complex function theory and weight theory, we discover new criteria of describing the boundedness and compactness of weighted composition operators acting on weighted Bergman spaces. Moreover, we are able to establish some new weighted type estimates for a new class of weights, which is adapted to Sawyer--testing conditions (see, Definition \ref{newweight} and Remark \ref{newweightrem}). Again, to the best of our knowledge, these types of results appear to be new in the literature, and more importantly, they seem not be covered by the classical Carleson measure technique.
\end{enumerate}

The structure of this paper is as follows. Section 2 provides backgrounds, especially the dyadic system and sparse family in $(\R_+^2, dA_\alpha)$, and Section 3 characterizes a standard Carleson embedding type theorem. In Section 4, we first give new necessary and sufficient conditions for the weighted composition operators to be bounded and compact on the weighted Bergman spaces. Moreover, we establish a new weighted type estimate, together with introducing a new class of weighs that is adapted to Sawyer's classical test conditions. In Section 5, we deal with the analog of our results in the unit ball $\B$ in $\C^n$, and finally, in Section 6, we give some remarks  for possible extensions of our main results.

Throughout this paper, for $a, b \in \R$, $a \lesssim b$ ($a \gtrsim b$, respectively) means there exists a positive number $C$, which is independent of $a$ and $b$, such that $a \leq Cb$ ($ a \geq Cb$, respectively). Moreover, if both $a \lesssim b$ and $a \gtrsim b$ hold, then we say $a \simeq b$.

% ----------------------------------------------------------------
\bigskip

\section{Preliminary}

In this section, we recall some basic facts from the dyadic calculus on $(\R^2_+, dA_\alpha)$. For $a=x_a+iy_a \in \R^2_+$, we denote
$$
T_{a}:=\left\{ z=x+iy \in \R^2_+: |x-x_a| \le \frac{y_a}{2}, 0<y<y_a \right\}
$$
to be the \emph{Carleson tent} associated to $a$. While for an interval $I \subset \R$, we denote
$$
Q_I:=\left\{ z=x+iy \in \R^2_+: x \in I, y<|I| \right\}
$$
to be the \emph{Carleson box} associated to $I$. We note that
$$
Q_I=T_{(c_I, |I|)},
$$
where $c_I$ is the center of $I$. For any $E \subseteq \R^2_+$, denote $A_\alpha(E):=\int_E dA_\alpha(z)$. Then it is easy to see that
$$
A_\alpha(T_a) \simeq y_a^{\alpha+2} \simeq (y_z+y_a)^{\alpha+2} \simeq |z-\bar{a}|^{\alpha+2}, \quad z \in T_a.
$$

It will be convenient for us to decompose $Q_I$ into a disjoint union of small rectangles. To do this, we introduce the following definition.

\begin{defn} \label{Whitneydecomp}
Let $I=[a, b) \subset \R$ and $Q_I$ be the Carleson box associated to $I$. For each $i \ge1, i \in \N$, we define the \emph{$i$-th generation of the upper Whitney rectangles associated to $I$} as
$$
\calW_{i, I}:=\left\{ \bigg[a+\frac{(b-a)(j-1)}{2^{i-1}}, a+\frac{(b-a)j}{2^{i-1}} \bigg) \times \bigg[\frac{b-a}{2^i}, \frac{b-a}{2^{i-1}} \bigg),  1 \le j \le 2^{i-1}  \right\}
$$
and the \emph{collection of upper Whiteney rectangles associated to $I$} as
$$
\calW_I:=\bigcup_{i \ge 1}^\infty \calW_{i, I}.
$$
In particular, there is only one rectangle in $\calW_{1, I}$, which is denoted as $Q_I^{\textrm{up}}$. Moreover,
$$
Q_I=\bigcup_{R \in \calW_I} R
$$
\end{defn}

We have the following lemma, which is an easy application of the mean value property of subharmonic function.

\begin{lem} \label{subharmonic}
Let $I=[a, b)$ and $Q_I$ be defined as above. Let further, $R \in \calW_{i, I}$ for some $i \ge 1$. Then for any $f \in H(\R^2_+)$,
$$
|f(z)| \lesssim \frac{1}{A_\alpha(R)} \int_{\frac{3R}{2}} |f(z)|dA_\alpha(z), \quad z \in R,
$$
where the implicit constant in the above inequality only depends on $\alpha$, and $\frac{3R}{2}$ is the dilation of $R$ with same center but with side lengths $3/2$ times of $R$.
\end{lem}

We make a remark that the ratio $3/2$ is not necessary in the above lemma. Indeed, any number in the range $(1, 3)$ works.

Next, we would like to extend the above constructions to a collection of intervals, namely, on a dyadic grid on $\R$.

\begin{defn} \label{dyadicgrid}
A collection of intervals $\calD$ in $\R$ is a \emph{dyadic grid} if the following statements hold:
\begin{enumerate}
\item [(i)] If $I \in \calD$, then $\ell(I)=2^k$ for some $k \in \Z$, where $\ell(I)$ refers to the sidelength of the interval $I$;
\item [(ii)] If $I, J \in \calD$, then $I \cap J \in \left\{I, J, \emptyset \right\}$;
\item [(iii)] For every $k \in \Z$, the intervals $\calD_k=\left\{I \in \calD: \ell(I)=2^k \right\}$ form a partition of $\R$.
\end{enumerate}
\end{defn}

This allows us to consider the collection of Carleson boxes induced by the dyadic grid $\calD$, which we denote as $Q_\calD$.

\begin{lem} \label{sparse}
Let $\calD$ and $Q_\calD$ be defined as above. Then there exists $0<\sigma<1$, such that for any $Q \in Q_\calD$,
$$
A_\alpha \left( \bigcup_{P \in Q_\calD, P \subsetneq Q} P \right) \le \sigma A_\alpha(Q).
$$
Equivalently, if we define
$$
E(Q)=Q \backslash \bigcup_{P \in Q_\calD, P \subsetneq Q} P,
$$
then the sets $E(Q)$ are pairwise disjoint and $A_\alpha(E(Q)) \ge (1-\sigma) A_\alpha(Q)$.
\end{lem}

\begin{proof}
This follows from an easy calculation and it suffices to take $\sigma=\frac{1}{2^{\alpha+1}}$. We leave the details to the reader.
\end{proof}

\begin{rem} \label{rem01}
Note that there is a natural way to embed $Q_\calD$ into a dyadic grid in $\R^2$, and therefore, Lemma \ref{sparse} asserts that $Q_\calD$ is a \emph{sparse collection} of some dyadic grid in $(\R^2_+, dA_\alpha)$ with sparseness $1-\sigma$.
\end{rem}

\begin{lem}[{\cite[Theorem 3.4]{DCU}}] \label{Meilemma}
There exist dyadic grids $\calD^1, \calD^2$ and $\calD^3$, such that for any interval $I$, there exists $J \in \calD^k$ for some $k \in \{1, 2, 3\}$, such that $I \subset J$ and $\ell(J) \le 3\ell(I)$.
\end{lem}

A possible choice for these three dyadic grids in $\R$ is
\begin{equation} \label{explicitconst}
\calD^k=\left\{ 2^j ( [0, 1)+m+t): j \in \Z, m \in \Z \right\}, t \in \{0, \rpm 1/3\}.
\end{equation}
From now on, we shall fix a choice of three dyadic grids $\calD^1$, $\calD^2$ and $\calD^3$, which satisfies the conclusion in Lemma \ref{Meilemma}.

% ----------------------------------------------------------------
\bigskip

\section{Carleson embedding}

The results in this section are standard, and to be self-contained, we include their proofs here. Recall that for $\lambda>0$,  we say a measure $\mu$ defined on $\R^2_+$ is a  \emph{$\lambda$-Carleson measure} if
$$
\sup_{a \in \R^2_+} \frac{\mu(T_a)}{\left(A_\alpha(T_a) \right)^\lambda} <\infty,
$$
and a \emph{vanishing $\lambda$-Carleson measure} if
$$
\lim_{a \to \partial( \widehat{\R^2_+} ) }\frac{\mu(T_a)}{\left(A_\alpha(T_a) \right)^\lambda}=0.
$$
Here $\lim\limits_{z \to \partial( \widehat{\R^2_+} ) } g(z)=0$ means that $\sup\limits_{\R^2_+ \backslash K} |g| \to 0$ as the compact set $K \subset \R^2_+$ expands to all of $\R^2_+$, or equivalently that $g(z) \to 0$ as $\im z \to 0^{+}$ and $g(z) \to 0$ as $|z| \to \infty$.

Given $p \ge 1, \alpha>-1$, $u\in H(\R^2_+)$ and $\varphi: \R^2_+ \rightarrow \R^2_+$ a holomorphic mapping, we define the measure $\mu_{u, \varphi, p, \alpha}$ by
$
\mu_{u, \varphi, p, \alpha}(E):=\left( |u|^pA_\alpha  \right) \left(\varphi^{-1}(E) \right).
$
Namely, for any $f$ measurable, we have
$$
\int_{\R^2_+} f d\mu_{u, \varphi, p, \alpha}= \int_{\R^2_+} f \circ \varphi(z) |u(z)|^p dA_\alpha(z).
$$

A simple, standard calculation yields the following lemma.
\begin{lem} \label{testfunction}
For any $a \in \R^2_+$ and $t \ge 1$, let
$
f_{a, t}(z):=\frac{y_a^{\frac{\alpha+2}{t}}}{(z-\bar{a})^{\frac{2\alpha+4}{t}}}, z \in \R^2_+.
$
Then $f_{a, t} \in A_\alpha^t$ and $\sup\limits_{a \in \R^2_+} \|f_{a, t}\|_{t, \alpha} \lesssim 1$.
\end{lem}

We have the following Carleson type result.

\begin{thm} \label{Carleson}
Let $q \ge p \ge 1$, $\alpha>-1$, $u\in H(\R^2_+)$ and $\varphi: \R^2_+ \rightarrow \R^2_+$ be a holomorphic mapping. Then the following statements are equivalent.
\begin{enumerate}
\item [(i)] $\mu_{u, \varphi, p, \alpha}$ is a $\frac{q}{p}$-Carleson measure;
\item [(ii)] $W_{u, \varphi}: A_\alpha^p \rightarrow A_\alpha^q$ is bounded;
\item [(iii)]  The following testing condition holds:
\begin{equation} \label{testcond}
\sup_{a \in \R^2_+} \int_{\R^2_+} \frac{|y_a|^{\frac{(\alpha+2)q}{p}}|u(z)|^q}{\left| \varphi(z)-\bar{a} \right|^{\frac{(2\alpha+4)q}{p}}} dA_\alpha(z)<\infty.
\end{equation}
\end{enumerate}
\end{thm}

\begin{proof} {(i) $\Longrightarrow$ (ii). }  Take and fix any dyadic grid $\calD$ on $\R$. Note that
$$
\R^2_+=\bigcup_{I \in \calD} Q_I^{\textrm{up}}.
$$
Therefore, for any $f \in A^q_\alpha$, by Lemma \ref{subharmonic}, (i) and the fact that $q \ge p$, we have
\begin{eqnarray*}
\|W_{u, \varphi} f\|_{q, \alpha}^q%
&=& \int_{\R^2_+} |u(z)|^q |f(\varphi(z))|^qdA_\alpha(z) =\int_{\R^2_+} |f(z)|^q d\mu_{u, \varphi, q, \alpha}(z) \\
&\le& \sum_{I \in \calD} \int_{Q_I^{\textrm{up}}} |f(z)|^q d\mu_{u, \varphi, q, \alpha}(z) \\
&\lesssim& \sum_{I \in \calD} \int_{Q_I^{\textrm{up}}} \frac{1}{A_\alpha(Q_I^{\textrm{up}})} \left( \int_{\frac{3}{2} Q_I^{\textrm{up}}} |f(w)|^p dA_\alpha(w) \right) d\mu_{u, \varphi, q, \alpha}(z)\\
&=& \sum_{I \in \calD} \frac{\mu_{u, \varphi, q, \alpha}(Q_I^{\textrm{up}})}{A_\alpha(Q_I^{\textrm{up}})} \cdot  \int_{\frac{3}{2} Q_I^{\textrm{up}}} |f(w)|^p dA_\alpha(w)
\end{eqnarray*} \begin{eqnarray*}
&\lesssim& \sum_{I \in \calD}  \int_{\frac{3}{2} Q_I^{\textrm{up}}} A_\alpha(Q_I^{\textrm{up}})^{\frac{q-p}{p}} |f(w)|^{q-p} |f(w)|^pdA_\alpha(w) \\
&\le& \|f\|_{p, \alpha}^{q-p} \sum_{I \in \calD} \int_{\frac{3}{2}Q_I^{\textrm{up}}} |f(w)|^pdA_\alpha(w) \\
&\lesssim& \|f\|_{p, \alpha}^q,
\end{eqnarray*}
where in the last inequality, we use the fact that the set $\left\{\frac{3}{2}Q_I^{\textrm{up}} \right\}_{I \in \calD}$ has finite overlap.

{(ii) $\Longrightarrow$ (iii).} This is straightforward by testing the functions $\{f_{a, p}\}_{a \in \R^2_+}$ in Lemma \ref{testfunction}.

 {(iii) $\Longrightarrow$ (i).} For each $a \in \R^2_+$, we have
\begin{eqnarray*}
1%
&\gtrsim& |y_a|^{\frac{(\alpha+2)q}{p}} \cdot \int_{\R^2_+} \frac{|u(z)|^q}{\left| \varphi(z)-\bar{a} \right|^{\frac{(2\alpha+4)q}{p}}} dA_\alpha(z) \\
&=& |y_a|^{\frac{(\alpha+2)q}{p}} \int_{\R^2_+} \frac{d\mu_{u, \varphi, q, \alpha}(z)}{|z-\bar{a}|^{\frac{(2\alpha+4)q}{p}}}  \ge   |y_a|^{\frac{(\alpha+2)q}{p}} \int_{T_a} \frac{d\mu_{u, \varphi, q, \alpha}(z)}{|z-\bar{a}|^{\frac{(2\alpha+4)q}{p}}} \\
&\simeq& \frac{\mu_{u, \varphi, q, \alpha}(T_a)}{ A_\alpha(T_a)^{\frac{q}{p}}},
\end{eqnarray*}
which implies the desired result.
\end{proof}

\begin{cor}
Let $q \ge p \ge 1$, $\alpha>-1$, $u\in H(\R^2_+)$ and $\varphi: \R^2_+ \rightarrow \R^2_+$ be a holomorphic mapping. If $W_{u, \varphi}: A_\alpha^p \rightarrow A_\alpha^q$ is bounded, then for any $\beta \in [p, q]$,  $W_{u, \varphi}: A_\alpha^\beta \rightarrow A_\alpha^q$ is also bounded.
\end{cor}

\begin{proof}
Let us first prove the result for those $\beta \in [p, (\alpha+2)p)$. By the boundedness of $W_{u, \varphi}: A_\alpha^p \rightarrow A_\alpha^q$,  we have
\begin{eqnarray*}
\frac{\mu_{u, \varphi, q, \alpha}(T_a)}{y_a^{(\alpha+2) \cdot \frac{2q}{\beta}}}%
& \lesssim& \int_{T_a} \frac{1}{|z-\bar{a}|^{{(\alpha+2) \cdot \frac{2q}{\beta}}}}  d\mu_{u, \varphi, q, \alpha} (z)\\
&\le& \int_{\R^2_+} \frac{1}{|z-\bar{a}|^{{(\alpha+2) \cdot \frac{2q}{\beta}}}}  d\mu_{u, \varphi, q, \alpha} (z) = \int_{\R^2_+} \frac{|u(z)|^q}{|\varphi(z)-\bar{a}|^{(\alpha+2) \cdot \frac{2q}{\beta}}} dA_\alpha(z) \\
&\lesssim&  \left( \int_{\R^2_+} \frac{1}{|z-\bar{a}|^{\frac{\left(2\alpha+4\right)p}{\beta}} } dA_\alpha(z) \right)^{\frac{q}{p}}.
\end{eqnarray*}
 Note that we can write $\frac{(2\alpha+4)p}{\beta}=2\alpha'+4$,  where $\alpha'=\frac{(\alpha+2)p}{\beta}-2>-1.$ Therefore,
$$
\frac{\mu_{u, \varphi, q, \alpha}(T_a)}{y_a^{(\alpha+2) \cdot \frac{2q}{\beta}}} \lesssim  \left( \int_{\R^2_+} \frac{1}{|z-\bar{a}|^{2\alpha'+4}} dA_\alpha(z) \right)^{\frac{q}{p}} \lesssim \frac{1}{y_a^{(\alpha'+2) \cdot \frac{q}{p}}}=\frac{1}{y_a^{(\alpha+2) \cdot \frac{q}{\beta}}},
$$
which implies that
$$
\mu_{u, \varphi, q, \alpha}(T_a) \lesssim y_a^{(\alpha+2) \cdot \frac{q}{\beta}} \simeq A_\alpha(T_a)^{\frac{q}{\beta}}.
$$

The general case follows from iterating the above argument with a larger ``$p$" each time. More precisely, from the above argument, we see that
$$
W_{u, \varphi}: A_\alpha^{(\alpha+2-\varepsilon)p} \to A_\alpha^q
$$
is bounded, for some $0<\varepsilon<\alpha+1$ (in particular, the choice of $\varepsilon$ only depends on $\alpha$). Then we rename ``$(\alpha+2-\varepsilon)p$" as our new ``$p$" and then iterate. Finally, we note that such iterations will stop when $\beta=q$, so that Theorem \ref{Carleson} (in particular, $(i) \Rightarrow (ii)$) applies.
\end{proof}

For the compactness of $W_{u, \varphi}$, we have the following result.

\begin{thm} \label{vanCarleson}
Let $q \ge p \ge 1$, $\alpha>-1$, $u\in H(\R^2_+)$ and $\varphi: \R^2_+ \rightarrow \R^2_+$ be a holomorphic mapping. Then the following statements are equivalent.
\begin{enumerate}
\item [(i)] $\mu_{u, \varphi, p, \alpha}$ is a vanishing $\frac{q}{p}$-Carleson measure;
\item [(ii)] $W_{u, \varphi}: A_\alpha^p \rightarrow A_\alpha^q$ is compact;
\item [(iii)] The following vanishing testing condition holds:
$$
\lim_{a \to \widehat{\R^2_+}} \int_{\R^2_+} \frac{|y_a|^{\frac{(\alpha+2)q}{p}}|u(z)|^q}{\left| \varphi(z)-\bar{a} \right|^{\frac{(2\alpha+4)q}{p}}} dA_\alpha(z)=0.
$$
\end{enumerate}
\end{thm}

\begin{proof}
The proof of this theorem is an easy modification of the proof of Theorem \ref{Carleson}, and hence we omit it.
\end{proof}

% ----------------------------------------------------------------
 \smallskip

\section{Sparse Domination of Weighted Composition Operators}

In this section, we study a sparse bound of a weighted composition operator $W_{u, \varphi}$ acting from $A^p_\alpha$ to $A^q_\alpha$, for some $q \ge p \ge 1$. Namely, we want to understand how one can study the quantity $\|W_{u, \varphi}\|_{q, \alpha}$ via only a sparse collection of cubes in $(\R^2_+, dA_\alpha)$ (see, Remark \ref{rem01}).

We need the following result on the integral representation of an $A^p_\alpha$ function.

\begin{lem} [{\cite[Theorem 1]{KLA}}] \label{integralreps} Let $1 \le p<\infty$. Then any function $f \in A^p_\alpha$ is representable in the form
$$
f(z)=C_\alpha \int_{\R^2_+} \frac{f(\zeta)}{\left(i\left(\bar{\zeta}-z\right)\right)^{\alpha+2}}dA_\alpha(\zeta),
$$
where $C_\alpha>0$ is an absolute constant.
\end{lem}

% ----------------------------------------------------------------
\medskip

\subsection{Boundedness} In the first part of this section, we study the boundedness of $W_{u, \varphi}$ by using the sparse domination technique.

Given any function $f \in A^p_\alpha$, we wish to understand the quantity $\|W_{u, \varphi} f\|_{q, \alpha}$. For any $N \in \N$ with $1 \le N \le p$, using Lemma \ref{integralreps}, we have 
\begin{eqnarray} \label{190212eq01}
\|W_{u, \varphi}f\|_{q, \alpha}^q%
&=& \int_{\R^2_+} |u(z)|^q |C_\varphi f(z)|^qdA_\alpha(z) \nonumber \\ 
&=& \int_{\R^2_+} |u(z)|^q |f \circ \varphi(z)|^{q-N} \left|C_\varphi (f^N)(z)\right| dA_\alpha(z) \nonumber \\
&\lesssim& \int_{\R^2_+} |u(z)|^q |f \circ \varphi(z)|^{q-N} \left( \int_{\R^2_+} \frac{|f(\zeta)|^N}{\left|\bar{\zeta}-\varphi(z) \right|^{\alpha+2}}dA_\alpha(\zeta) \right) dA_\alpha(z) \nonumber \\
&=& \int_{\R^2_+} |f(\zeta)|^N \left( \int_{\R^2_+} \frac{|f \circ \varphi(z)|^{q-N} |u(z)|^q dA_\alpha(z)}{\left|\bar{\zeta}-\varphi(z) \right|^{\alpha+2}} \right) dA_\alpha(\zeta) \nonumber \\
&=& \int_{\R^2_+} |f(\zeta)|^N \left( \int_{\R^2_+} \frac{|f(z)|^{q-N}}{|\bar{\zeta}-z|^{\alpha+2}}d\mu_{u, \varphi, q, \alpha}(z)  \right) dA_\alpha(\zeta).
\end{eqnarray}

To bound the term \eqref{190212eq01}, we have the following lemma, which can be viewed as an upper half plane analog of \cite[Lemma 5]{RTW}.

\begin{lem} \label{sparse-1}
Let $q \ge p \ge 1, \alpha>-1$,  $N \in \N$ with $1 \le N \le p$ and $\mu$ be a $\frac{q}{p}$-Carleson measure. Then for any $\gamma \ge 1$ and $\zeta \in \R^2_+$,
\begin{eqnarray*}
&& \int_{\R^2_+} \frac{|f(z)|^{q-N}}{\left| \overline{\zeta}-z \right|^{\alpha+2}} d\mu(z) \lesssim \sum_{i=1}^3 \sum_{I \in \calD^i} \frac{\one_{Q_I}(\zeta)}{A_\alpha(Q_I)} \cdot A^{\left( \frac{q}{p}-1 \right) \frac{1}{\gamma}}_\alpha(Q_I) \\
&&  \quad \quad \quad  \quad \quad \quad    \quad \quad \quad \quad \quad \quad  \quad \quad  \quad  \cdot \mu (Q_I)^{\frac{1}{\gamma'}} \left( \int_{Q_I} |f(z)|^{\gamma(q-N)} dA_\alpha(z) \right)^{\frac{1}{\gamma}}.
\end{eqnarray*}
\end{lem}

\begin{proof}
 For each $z \in \R^2_+$, we first consider the interval
$$
I_{z, \zeta}:=\bigg[ \frac{x_\zeta+x_z}{2}-|\bar{\zeta}-z|, \frac{x_\zeta+x_z}{2}+|\bar{\zeta}-z| \bigg) \subset \R.
$$
It is easy to see the following facts:
\begin{enumerate}
\item [(1)] $z, \zeta \in Q_{I_{z, \zeta}}$;
\item [(2)] $A_\alpha(Q_{I_{z, \zeta}}) \simeq |\bar{\zeta}-z|^{\alpha+2}$.
\end{enumerate}
Next, by Lemma \ref{Meilemma}, we are able to find an interval $I \in \calD^k$ for some $k \in \{1, 2, 3\}$, such that $I_{z, \zeta} \subset I$ and $\ell(I) \le 3\ell(I_{z, \zeta})$, which implies
\begin{enumerate}
\item [(3)] $z, \zeta \in Q_I$;
\item [(4)] $A_\alpha(Q_I) \simeq |\bar{\zeta}-z|^{\alpha+2}$.
\end{enumerate}
Therefore,
\begin{equation} \label{1004}
 \int_{\R^2_+} \frac{|f(z)|^{q-N}}{|\bar{\zeta}-z|^{\alpha+2}}d\mu(z) \lesssim \sum_{i=1}^3 \sum_{I \in \calD^i} \frac{\one_{Q_I}(\zeta)}{A_\alpha(Q_I)} \int_{Q_I} |f(z)|^{q-N}d\mu(z).
\end{equation}
Next, we claim that
\begin{eqnarray}  \label{190312eq01}
 \int_{Q_I} |f(z)|^{q-N}d\mu(z) %
& \lesssim& \mu (Q_I)^{\frac{1}{\gamma'}} \cdot A_\alpha^{\left(\frac{q}{p}-1\right) \frac{1}{\gamma} }(Q_I)  \nonumber \\
&& \quad \quad \quad \quad \quad \cdot \left( \int_{\frac{3}{2}Q_I \cap \R^2_+} |f(z)|^{\gamma(q-N)} dA_\alpha(z) \right)^{\frac{1}{\gamma}}.
\end{eqnarray}
Indeed, by using H\"older's inequality, it suffices to bound the term
\begin{equation} \label{20191108eq01}
 \int_{Q_I} |f(z)|^{\gamma(q-N)}d\mu(z).
\end{equation}
We further decompose the cube $Q_I$ into its upper Whitney rectangles. More precisely, using Lemma \ref{subharmonic}, the fact that $|f(z)|^{\gamma(q-N)}$ is a subharmonic function and Theorem \ref{Carleson}, we have
\begin{eqnarray*}
\eqref{20191108eq01} %
&=& \sum_{R \in \calW_I} \int_R |f(z)|^{\gamma(q-N)} d\mu(z) \\
&\le& \sum_{R \in \calW_I} \frac{\mu(R)}{A_\alpha(R)} \int_{\frac{3R}{2}} |f(z)|^{\gamma(q-N)} dA_\alpha(z)\\
&\lesssim& \sum_{R \in \calW_I} A_\alpha^{\frac{q}{p}-1}(R)\int_{\frac{3R}{2}} |f(z)|^{\gamma(q-N)} dA_\alpha(z)\\
&\le& \left(\sup_{R \in \calW_I} A_\alpha^{\frac{q}{p}-1}(R) \right) \cdot \sum_{R \in \calW_I} \int_{\frac{3R}{2}} |f(z)|^{\gamma(q-N)} dA_\alpha(z)\\
&\lesssim& A_\alpha^{\frac{q}{p}-1}(Q_I) \int_{\frac{3}{2}Q_I \cap \R^2_+} |f(z)|^{\gamma(q-N)}dA_\alpha(z),
\end{eqnarray*}
where in the last inequality, we use the fact that the set $\left\{ \frac{3R}{2} \right\}_{R \in \calW_I}$ has finite overlap. The desired claim follows from by pluging the above estimate to \eqref{20191108eq01}.

Combining \eqref{1004} and \eqref{190312eq01}, we get
\begin{eqnarray}  \label{190312eq04}
 &&\int_{\R^2_+} \frac{|f(z)|^{q-N}}{|\bar{\zeta}-z|^{\alpha+2}}d\mu(z)  \lesssim \sum_{i=1}^3 \sum_{I \in \calD^i} \frac{\one_{Q_I}(\zeta)}{A_\alpha(Q_I)} \cdot \mu (Q_I)^{\frac{1}{\gamma'}} \nonumber \\
&&  \quad \quad \quad \quad \quad  \quad \quad \quad   \quad \quad  \cdot A_\alpha(Q_I)^{\left(\frac{q}{p}-1 \right) \frac{1}{\gamma}} \left(\int_{\frac{3}{2}Q_I \cap \R^2_+} |f(z)|^{\gamma(q-N)}dA_\alpha(z) \right)^{\frac{1}{\gamma}}.
\end{eqnarray}
We wish to change the integral domain in the above integration from $\frac{3}{2}Q_I \cap \R^2_+$ to a dyadic cube belonging to $Q_{\calD^i}$ for some $i \in \{1, 2, 3\}$. To see this, we apply Lemma \ref{Meilemma} again. More precisely, since $\overline{\left(\frac{3}{2}Q_I \cap \R^2_+\right)} \bigcap \R=\frac{3I}{2}$, using Lemma \ref{Meilemma}, we can take $J \in \calD^i$ for some $i \in \{1, 2, 3\}$, such that
\begin{equation} \label{190312eq02}
I \subset \frac{3I}{2} \subset J \ \textrm{and} \ \ell(J) \le 3 \ell\left(\frac{3I}{2}  \right) \le 6 \ell(I).
\end{equation}
This suggests that we have the pointwise bound
\begin{eqnarray} \label{190312eq03}
&& \frac{\one_{Q_I}(\zeta) \mu (Q_I)^{\frac{1}{\gamma'}}}{A_\alpha(Q_I)}   A_\alpha(Q_I)^{ \left(\frac{q}{p}-1 \right)\frac{1}{\gamma}} \left(\int_{\frac{3}{2}Q_I \cap \R^2_+} |f(z)|^{\gamma(q-N)}dA_\alpha(z) \right)^{\frac{1}{\gamma}} \nonumber \\
&&   \lesssim\frac{\one_{Q_J}(\zeta) \mu (Q_J)^{\frac{1}{\gamma'}}}{A_\alpha(Q_J)} A_\alpha(Q_J)^{\left(\frac{q}{p}-1 \right)\frac{1}{\gamma}} \left(\int_{Q_J} |f(z)|^{\gamma(q-N)}dA_\alpha(z) \right)^{\frac{1}{\gamma}},
\end{eqnarray}
where we use the fact that $A_\alpha(Q_I) \simeq A_\alpha(Q_J)$. Finally, we need to check that each $J \in \calD^i, i \in \{1, 2, 3\}$,  only appears finitely many times when we apply the inequality \eqref{190312eq03}. Indeed, this is clear from \eqref{190312eq02} and the dyadic structure on $\R$.
The desired result then follows from \eqref{190312eq04} and \eqref{190312eq03}.
\end{proof}

For any set $E \subset \R^2_+$, $\gamma \ge 1$ and $g \ge 0$ on $\R^2_+$, we set
  $$\langle g \rangle_{E, \gamma}:=\left(\frac{1}{A_\alpha(E)} \int_E |g(z)|^\gamma dA_\alpha(z) \right)^{\frac{1}{\gamma}}.
$$

From \eqref{190212eq01}, Lemma 4.2 and Theorem \ref{Carleson}, we have the following result.

\begin{prop} \label{20190914prop01} Let $q \ge p \ge 1, \alpha>-1$, $u\in H(\R^2_+)$ and $\varphi: \R^2_+ \rightarrow \R^2_+$ be a holomorphic mapping. Let further, $W_{u, \varphi}: A^p_\alpha \mapsto A^q_\alpha$ be bounded. Then for any $\gamma \ge 1$ and $f \in A^p_\alpha$,
\begin{equation}  \label{010101}
\|W_{u, \varphi} f\|_{q, \alpha}^q \lesssim \inf\limits_{N \in \N,  1 \le N \le p} \left( \sum_{i=1}^3 \sum_{I \in \calD^i} \mu_{u, \varphi, q, \alpha} (Q_I)^{\frac{1}{\gamma'}} A_\alpha(Q_I)^{\frac{q}{p\gamma}} \langle |f|^N \rangle_{Q_I} \cdot  \langle |f|^{q-N} \rangle_{Q_I, \gamma} \right).
\end{equation}
\end{prop}

In particular, when $p=q$, we have the following result.

\begin{thm} \label{190313cor01} Let $q \ge 1, \alpha>-1$, $u\in H(\R^2_+)$ and $\varphi: \R^2_+ \rightarrow \R^2_+$ be a holomorphic mapping. Then the following statements are equivalent.
\begin{enumerate}
\item [(i)] $W_{u, \varphi}: A^q_\alpha \mapsto A^q_\alpha$ is bounded;
\item [(ii)] For any $f \in A^q_\alpha$,
$$
\|W_{u, \varphi} f\|_{q, \alpha}^q \lesssim \inf\limits_{N \in \N,  1 \le N \le q} \left( \sum_{i=1}^3 \sum_{I \in \calD^i} A_\alpha(Q_I) \langle |f|^N \rangle_{Q_I} \cdot  \langle |f|^{q-N} \rangle_{Q_I} \right).
$$
\end{enumerate}
\end{thm}

\begin{proof}
The assertion $(ii)$ implies $(i)$ follows from Proposition \ref{20190914prop01} with $\gamma=1$, and therefore we only need to show that (ii) implies (i). Without the loss of generality, we fix a $N \in \N$ with $1 \le N<q$ (the case $q=N$ follows from a similar argument, and we leave the detail to the interested reader), then we can find some $i_0 \in \{1, 2, 3\}$, such that
$$
 \sum_{i=1}^3 \sum_{I \in \calD^i} A_\alpha(Q_I) \langle |f|^N \rangle_{Q_I} \cdot  \langle |f|^{q-N} \rangle_{Q_I} \le 3\sum_{I \in \calD^{i_0}} A_\alpha(Q_I) \langle |f|^N \rangle_{Q_I} \cdot  \langle |f|^{q-N} \rangle_{Q_I}.
$$
Therefore,
\begin{eqnarray*}
\|W_{u, \varphi}f\|_{q, \alpha}^q%
& \lesssim& \sum_{I \in \calD^{i_0}} A_\alpha(Q_I) \langle |f|^N \rangle_{Q_I} \cdot  \langle |f|^{q-N} \rangle_{Q_I} \\
& \lesssim& \sum_{I \in \calD^{i_0}} A_\alpha(Q_I^{\textrm{up}})  \langle |f|^N \rangle_{Q_I} \cdot  \langle |f|^{q-N} \rangle_{Q_I} \\
& \le& \int_{\R^2_+} \calM\left(|f|^N\right) \calM \left(|f|^{q-N} \right) dA_\alpha(z) \\
& \le & \left( \int_{\R^2_+} \calM(|f|^N)^{\frac{q}{N}} dA_\alpha(z) \right)^{\frac{N}{q}} \cdot \left( \int_{\R^2_+} \calM(|f|^{q-N})^{\frac{q}{q-N}}dA_\alpha(z) \right)^{\frac{q-N}{q}}\\
&\lesssim& \left( \int_{\R^2_+} |f|^{N \cdot \frac{q}{N}} dA_\alpha(z) \right)^{\frac{N}{q}} \cdot \left( \int_{\R^2_+} |f|^{(q-N) \cdot \frac{q}{q-N}} dA_\alpha(z) \right)^{\frac{q-N}{q}}\\
&=& \|f\|_{q, \alpha}^q,
\end{eqnarray*}
where in the above estimates, $\calM$ is the usual uncentered Hardy-Littlewood maximal operator with respect to the measure $A_\alpha$, and in the last inequality, we use a classical fact that $\calM$ is a bounded operator from $L^r_\alpha$ to itself, for $1<r \le \infty$.
\end{proof}

We can also establish such an equivalence for the case when $p<q$ with some extra assumptions.

\begin{thm} \label{190313cor02} Let $2p>q > p \ge 1, \alpha>-1$, $u\in H(\R^2_+)$ and $\varphi: \R^2_+ \rightarrow \R^2_+$ be a holomorphic mapping. Suppose
$$
\Z_{p, q}:=\left\{ N \in \N:  ~~~~ N \ge 1, N<p<q<p+N \right\} \neq \emptyset.
$$
Then the following statements are equivalent:
\begin{enumerate}
\item [(i).] $W_{u, \varphi}: A^p_\alpha \mapsto A^q_\alpha$ is bounded;
\item [(ii).] For any $f \in A^p_\alpha$,
$$
\|W_{u, \varphi} f\|_{q, \alpha}^q \lesssim \inf\limits_{N \in \Z_{p, q}} \left( \sum_{i=1}^3 \sum_{I \in \calD^i} A^{\frac{q}{p}}_\alpha(Q_I) \langle |f|^N \rangle_{Q_I} \cdot  \langle |f|^{q-N} \rangle_{Q_I} \right).
$$
\end{enumerate}
\end{thm}

We make a remark that  in general $\Z_{p, q}$ is not trivial, one typical example for $\Z_{p, q}$ to be non-empty is that both $p, q$ are large but $q-p$ is small.

\begin{proof}
The idea of proof of this result follows from the proof of Theorem \ref{190313cor01}, and the new ingredient in this proof is that instead of using the Hardy-Littlewood maximal function, we use its fractional version. Again, we only need to show that (ii) implies (i). First we note that our assumption $p<q<2p$ implies
$0<\frac{2q}{p}-2<2.$ Write
$$
l=\frac{p}{p+N-q} \quad \textrm{and} \quad l'=\frac{p}{q-N}.
$$
Fix any $N \in \Z_{p, q}$. Then a simple calculation yields
\begin{equation} \label{190313eq01}
\frac{N}{p}-\frac{1}{l}=\frac{q}{p}-1.
\end{equation}
Also note that $1<\frac{p}{N}<\frac{p}{q-p}.$  Let $i_0 \in \{1, 2, 3\}$ be the index such that
$$
 \sum_{i=1}^3 \sum_{I \in \calD^i} A^{\frac{q}{p}}_\alpha(Q_I) \langle |f|^N \rangle_{Q_I} \cdot  \langle |f|^{q-N} \rangle_{Q_I} \le 3\sum_{I \in \calD^{i_0}} A^{\frac{q}{p}}_\alpha(Q_I) \langle |f|^N \rangle_{Q_I} \cdot  \langle |f|^{q-N} \rangle_{Q_I}.
$$
Therefore,
\begin{eqnarray*}
\|W_{u, \varphi}f\|_{q, \alpha}^q%
& \lesssim& \sum_{I \in \calD^{i_0}} A^{\frac{q}{p}}_\alpha(Q_I) \langle |f|^N \rangle_{Q_I} \cdot  \langle |f|^{q-N} \rangle_{Q_I} \\
&\lesssim& \sum_{I \in \calD^{i_0}} A_\alpha(Q_I^{\textrm{up}}) \cdot \frac{A^{\frac{q}{p}-1}_\alpha(Q_I)}{A_\alpha(Q_I)} \int_{Q_I} |f|^N dA_\alpha(z) \cdot  \langle |f|^{q-N} \rangle_{Q_I} \\
&=& \sum_{I \in \calD^{i_0}} A_\alpha(Q_I^{\textrm{up}}) \cdot \frac{A^{\frac{\frac{2q}{p}-2}{2}}_\alpha(Q_I)}{A_\alpha(Q_I)} \int_{Q_I} |f|^N dA_\alpha(z) \cdot  \langle |f|^{q-N} \rangle_{Q_I}\\
& \le& \int_{\R^2_+} \calM_{\frac{2q}{p}-2} \left(|f|^N\right) \calM \left(|f|^{q-N} \right) dA_\alpha(z) \\
& \le & \left( \int_{\R^2_+} \calM_{\frac{2q}{p}-2}(|f|^N)^l dA_\alpha(z) \right)^{\frac{1}{l}} \cdot \left( \int_{\R^2_+} \calM(|f|^{q-N})^{l'}dA_\alpha(z) \right)^{\frac{1}{l'}}\\
&\lesssim& \left( \int_{\R^2_+} |f|^{N \cdot \frac{p}{N}} dA_\alpha(z) \right)^{\frac{N}{p}} \cdot \left( \int_{\R^2_+} |f|^{l'(q-N) } dA_\alpha(z) \right)^{\frac{1}{l'}}\\
&=& \|f\|_{p, \alpha}^q.
\end{eqnarray*}
Here in the above estimates, $\calM_{\frac{2q}{p}-2}$ is the fractional Hardy-Littlewood maximal operator with respect to the measure $A_\alpha$, and in the last inequality, we use the fact that
$$
\calM_{\frac{2q}{p}-2}: L^{\frac{p}{N}}_\alpha \mapsto L^l_\alpha
$$
is bounded, which is guaranteed by \eqref{190313eq01}.
\end{proof}

% ----------------------------------------------------------------
\bigskip

\subsection{Compactness} In the second part of this section, we establish a new characterization of the compactness of $W_{u, \varphi}$ via sparse domination.

Recall in the previous part, we are able to capture the boundedness of $W_{u, \varphi}:A_\alpha^p \mapsto A_\alpha^q$ by using the sparse form
\begin{equation} \label{190314eq01}
\sum_{I \in \calD} A_\alpha^{\frac{q}{p}}(Q_I) \langle |f|^N \rangle_{Q_I}  \cdot \langle |f|^{q-N} \rangle_{Q_I},
\end{equation}
for some $N \in \N, 0<N<q$ and some dyadic grid $\calD$. Note that this sparse form corresponds to the case $\gamma=1$ in Lemma \ref{sparse-1}. The interesting feature for this quantity is that it is independent of  the terms $u$ and $\varphi$. This suggests us that \eqref{190314eq01} may not be enough to describe the compactness of $W_{u, \varphi}$, which is clearly stronger than the boundedness. The idea is to consider those sparse forms in Lemma \ref{sparse-1} with $\gamma>1$.

The following is our main result for the compactness of $W_{u, \varphi}: A^q_\alpha \mapsto A^q_\alpha$.

\begin{thm} \label{compactness} Let $q \ge 1, \alpha>-1$, $u\in H(\R^2_+)$ and $\varphi: \R^2_+ \rightarrow \R^2_+$ be a holomorphic mapping.  If $W_{u, \varphi}: A^q_\alpha \mapsto A^q_\alpha$ is bounded,  then the following statements are equivalent.
\begin{enumerate}
\item [(i)] $W_{u, \varphi}: A^q_\alpha \mapsto A^q_\alpha$ is compact;
\item [(ii)]  Let $1 \le N<q, N \in \N$ and $1<\gamma<\frac{q}{q-N}$, or $q=N \in \N$ and $\gamma>1$. Let further, $\{K_n\}_{n \ge 1}$ be a sequence of sets exhausting $\R^2_+$, that is, $\{K_n\}_{n \ge 0}$ is a collection of compact sets in $\R^2_+$, satisfying $K_1 \subsetneq K_2 \subsetneq \dots K_n \subsetneq \dots \subsetneq \R^2_+,$
and $\bigcup\limits_{n=1}^\infty K_n=\R^2_+.$ Then for any bounded set $\{f_m\}_{m \ge 1} \subset A^p_\alpha$ with $f_m\rightarrow 0$ as $m\rightarrow\infty$, uniformly on compact subsets of $\R^2_+$,
$$
\lim_{n \to \infty}  \sup_{m \ge 1} \left(\sum_{i=1}^3 \sum_{I \in \calD^i, Q^{\textrm{up}}_I \cap K_n=\emptyset} \mu_{u, \varphi, q, \alpha} (Q_I)^{\frac{1}{\gamma'}}  A^{\frac{1}{\gamma}}_\alpha(Q_I) \langle |f_m|^N \rangle_{Q_I} \langle |f_m|^{q-N} \rangle_{Q_I, \gamma} \right)=0.
$$
\end{enumerate}
\end{thm}

\begin{proof} {   (i)$ \Rightarrow$(ii). }  Suppose $W_{u, \varphi}$ is compact, and hence by Theorem \ref{vanCarleson}, $\mu_{u, \varphi, p, \alpha}$ is a vanishing Carleson measure.  Write
$$
M:=\sup_{m \ge 1} \|f_m\|_{q, \alpha}^q.
$$
Then for any $\varepsilon>0$, there exists a $N_0 \in \N$, such that for $n>N_0$, we have
$$
\frac{\mu_{u, \varphi, q, \alpha}(T_a)}{A_\alpha(T_a)}<\left(\frac{\varepsilon}{M} \right)^{\gamma'},
$$
where $T_a$ is the Carleson tent associated to $a \in \R^2_+$ with satisfying  $T^{\textrm{up}}_a \cap K_n=\emptyset$ .

 Fix such an $n$. Then for any $i \in \{1, 2, 3\}$ and $m \ge 1$,
\begin{eqnarray*}
&& \sum_{I \in \calD^i, Q^{\textrm{up}}_I \cap K_n=\emptyset} \mu_{u, \varphi, q, \alpha} (Q_I)^{\frac{1}{\gamma'}}  A^{\frac{1}{\gamma}}_\alpha(Q_I) \langle |f_m|^N \rangle_{Q_I} \langle |f_m|^{q-N} \rangle_{Q_I, \gamma} \\
&&  \quad   \lesssim \frac{\varepsilon}{M} \cdot \sum_{I \in \calD^i, Q^{\textrm{up}}_I \cap K_n=\emptyset}   A_\alpha(Q_I) \langle |f_m|^N \rangle_{Q_I} \langle |f_m|^{q-N} \rangle_{Q_I, \gamma} \\
&&   \quad  \lesssim \frac{\varepsilon}{M} \cdot \sum_{I \in \calD^i}   A_\alpha(Q_I) \langle |f_m|^N \rangle_{Q_I} \langle |f_m|^{q-N} \rangle_{Q_I, \gamma} \\
&& \quad  \lesssim\frac{\varepsilon}{M}  \cdot \|f_m\|_{q, \alpha}^q \le \varepsilon.
\end{eqnarray*}
Here in the last inequality, we use the proof of Theorem \ref{190313cor01}. The desired result  follows by taking the supremem in $m$ first and then letting $\varepsilon$ converges to $0$.

\medskip

 {  (ii) $ \Rightarrow$(i). }  Without the loss of generality, we may consider the case $1 \le N<q$ and $1<\gamma<\frac{q}{q-N}$, while the proof of the case when $q=N$ and $\gamma>1$ is much more easier and we leave the detail to the interested reader. 
 
 Let $\{f_m\}_{m \ge 0} \subset A_\alpha^p$ be a bounded set satisfying $f_m \to 0$ as $m \to \infty$, uniformly on compact subsets on $\R^2_+$. It is well known that to prove $W_{u, \varphi}$ is compact, it suffices to show
$$
\|W_{u, \varphi} f_m\|_{q, \alpha} \to 0 \quad \textrm{as} \quad  m \to \infty.
$$
Let $\varepsilon>0$. By Lemma \ref{sparse-1} and without the loss of generality, we may assume
$$
\|W_{u, \varphi} f_m\|_{q, \alpha}^q \lesssim \sum_{I \in \calD^1} \mu_{u, \varphi, q, \alpha} (Q_I)^{\frac{1}{\gamma'}}  A^{\frac{1}{\gamma}}_\alpha(Q_I) \langle |f_m|^N \rangle_{Q_I} \langle |f_m|^{q-N} \rangle_{Q_I, \gamma}.
$$
Therefore, for each $m \ge 1$ and $n \ge 1$,
\begin{eqnarray*}
\|W_{u, \varphi} f_m\|_{q, \alpha}^q%
&\lesssim&\sum_{I \in \calD^1, Q_I^{\textrm{up}} \cap K_n=\emptyset} \mu_{u, \varphi, q, \alpha} (Q_I)^{\frac{1}{\gamma'}}  A^{\frac{1}{\gamma}}_\alpha(Q_I) \langle |f_m|^N \rangle_{Q_I} \langle |f_m|^{q-N} \rangle_{Q_I, \gamma} \\
&&+ \sum_{I \in \calD^1, Q_I^{\textrm{up}} \cap K_n \neq \emptyset} \mu_{u, \varphi, q, \alpha} (Q_I)^{\frac{1}{\gamma'}}  A^{\frac{1}{\gamma}}_\alpha(Q_I) \langle |f_m|^N \rangle_{Q_I} \langle |f_m|^{q-N} \rangle_{Q_I, \gamma} \\
&=:& A_{1, m}+A_{2, m}.
\end{eqnarray*}
Take any $\varepsilon>0$. We estimate $A_{2, m}$ first,  which is the main part. Put
$$
K'_n:=\overline{ \left\{ (x, y) \in \R^2_+: (x, y) \in Q^{\textrm{up}}_I, Q_I^{up} \cap K_n \neq \emptyset \right\}}.
$$
It is clear that $K'_n$ is a compact set and the collection $\{K_n'\}_{n \ge 1}$ is also a sequence of exhausting sets of $\R^2_+$.
Thus, for any $n \in \N$,
\begin{eqnarray*}
&&A_{2, m}= \sum_{I \in \calD^1, Q_I^{\textrm{up}} \cap K_n \neq \emptyset} \mu_{u, \varphi, q, \alpha} (Q_I)^{\frac{1}{\gamma'}}  A^{\frac{1}{\gamma}}_\alpha(Q_I) \langle |f_m|^N \rangle_{Q_I} \langle |f_m|^{q-N} \rangle_{Q_I, \gamma} \\
&\lesssim& \sum_{I \in \calD^1, Q_I^{\textrm{up}} \cap K_n \neq \emptyset}  A_\alpha(Q_I) \langle |f_m|^N \rangle_{Q_I} \langle |f_m|^{q-N} \rangle_{Q_I, \gamma} \\
&& \quad \textrm{(Since $\mu_{u, \varphi, q, \alpha}$ is a Carleson measure)} \\
&\lesssim& \sum_{I \in \calD^1, Q_I^{\textrm{up}} \cap K_n \neq \emptyset}  A_\alpha\left(Q^{\textrm{up}}_I\right) \langle |f_m|^N \rangle_{Q_I} \langle |f_m|^{q-N} \rangle_{Q_I, \gamma} \\
&\le& \int_{K'_n} M\left( |f_m|^N \right) \left( M \left( \left|f_m \right|^{(q-N) \gamma} \right) \right)^{\frac{1}{\gamma}} dA_\alpha(z) \\
&\le& \left( \int_{K'_n} M\left( |f_m|^N \right)^{\frac{q}{N}} dA_\alpha(z) \right)^{\frac{N}{q}} \left( \int_{K'_n} \left( M \left( \left|f_m \right|^{(q-N) \gamma} \right) \right)^{\frac{q}{\gamma(q-N)}}dA_\alpha(z) \right)^{\frac{q-N}{q}}.
\end{eqnarray*}
By H\"older's inequality, we get
\begin{eqnarray*}
A_{2, m} &\lesssim& \left( \int_{K'_n} |f_m|^{N  \cdot \frac{q}{N}} dA_\alpha(z) \right)^{\frac{N}{q}} \left( \int_{K'_n}  \left|f_m \right|^{(q-N) \gamma \cdot \frac{q}{\gamma(q-N)}}dA_\alpha(z) \right)^{\frac{q-N}{q}}\\
&=& \int_{K'_n} |f_m|^qdA_\alpha,
\end{eqnarray*}
where in the last estimate, we used the fact that $1<\gamma<\frac{q}{q-N}$ and the measure $\one_{K_n'}dA_\alpha$ is doubling. This implies that for any $n \ge 0$, we can take $m$ large enough, such that
\begin{equation} \label{estimateA2}
A_{2, m}< \frac{\varepsilon}{2}.
\end{equation}
Fix such a $m$. Then by the assumption, there exists an $N_0 \in \N$, such that for any $n>N_0$,
$$
\sum_{I \in \calD^1, Q_I^{\textrm{up}} \cap K_n=\emptyset} \mu_{u, \varphi, q, \alpha} (Q_I)^{\frac{1}{\gamma'}}  A^{\frac{1}{\gamma}}_\alpha(Q_I) \langle |f_m|^N \rangle_{Q_I} \langle |f_m|^{q-N} \rangle_{Q_I, \gamma}<\frac{\varepsilon}{2},
$$
which implies
$$
A_{1, m}<\frac{\varepsilon}{2}.
$$
The desired result then follows from the above estimate and \eqref{estimateA2}.
\end{proof}

\subsection{New weighted estimates} In the third part of this section, we apply the idea of sparse domination to obtain some new weighted estimates.

To start with, we recall that  by a \emph{weight} we will mean a  function  $\omega$ that is non-negative on a set of positive measure.

Let us introduce a new class of weights, which we denote as $\textbf{B}_{u, \varphi}^{\alpha, q}$ (Here, $\textbf{B}$ refers to a \emph{``Bergman projection''--like transformation} and this would be clear from the Definition \ref{newweight} below).

\begin{defn} \label{newweight}
Given $\alpha>-1$, $q>1$, $u \in H(\R^2_+)$ and $\varphi: \R_+^2 \to \R_+^2$, the weight class $\textbf{B}_{u, \varphi}^{\alpha, q}$ is defined to be the collection of all weights $\omega$ on $\R_+^2$ satisfying
\begin{equation} \label{Bpcondition}
\left[\omega \right]_{\textbf{B}_{u, \varphi}^{\alpha, q}}:=\sup_{\zeta \in \R_+^2} \int_{\R^2_+} \frac{|u(z)|^q \omega(z)}{\left| \bar{\zeta}-\varphi(z) \right|^{\alpha+2}} dA_\alpha(z)<\infty. \nonumber
\end{equation}
\end{defn}

\begin{rem} \label{newweightrem}
Note that the measure $wdA_\alpha$ itself may not be a Carleson measure. For example, let $u(z)=1, \varphi(z)=z+i$, $\alpha=0$, $q>1$ be any real number and
$$
w(z)=
\begin{cases}
0, \hfill \quad  |z| \ge 1, \ \im z>0; \\
1/|z|, \hfill \quad  |z|<1, \ \im z>0.
\end{cases}
$$

\medskip

\textbf{Claim 1: $C_\varphi$ is a bounded operator on $A^q$ (that is, $A^q_0$).}

\medskip

By Theorem \ref{Carleson}, it suffices to show
$$
\sup_{a \in \R^2_+} \int_{\R^2_+} \frac{|y_a|^2}{|z+i-\bar{a}|^4}dA(z)<\infty.
$$
This is clear as for each $a \in \R^2_+$, we have
\begin{eqnarray*}
\int_{\R^2_+} \frac{|y_a|^2}{|z+i-\bar{a}|^4}dA(z)%
&=&\int_\R \int_0^\infty \frac{|y_a|^2}{ \left(|x-x_a|^2+\left(y+1+y_a\right)^2 \right)^2} dydx \\
&=&  \int_\R \int_0^\infty  \frac{|y_a|^2}{ \left(x^2+\left(y+1+y_a\right)^2 \right)^2}dxdy \\
&\lesssim& 1.
\end{eqnarray*}
In particular, this implies the Carleson measure induced by the operator $C_\varphi$ is a $1$-Carleson measure.

\medskip

\textbf{Claim 2: $w \in \textbf{B}_{1, z+i}^{0, q}$.}

\medskip

 Indeed, for any $\zeta \in \R^2_+$, we have
$$
\int_{\R^2_+} \frac{|u(z)|^q \omega(z)}{\left| \bar{\zeta}-\varphi(z) \right|^{\alpha+2}} dA_\alpha(z)= \int_{|z|<1, \im z>0} \frac{1}{|z||\bar{\zeta}-z-i|^2}dA(z)  \lesssim 1,
$$
where in the last inequality, we first use the fact that $\frac{1}{|\bar{\zeta}-z-i|^2} \le 1$ for $z, \xi \in \R^2_+$ and then integrate in polar coordinate.

\medskip

\textbf{Claim 3: the measure $wdA$ is not $1$-Carleson.}

\medskip

Indeed
$$
w(T_a)=\int_{T_a} w(z)dA(z) \gtrsim \frac{1}{|y_a|} A(T_a),
$$
where $a=y_a i$ with $y_a>0$ sufficiently small. This implies that $\frac{w(T_a)}{A(T_a)} \to \infty \quad \textrm{as} \quad y_a \to 0$, which implies the desired claim.

 Moreover, we would also like to make a comment that this new class of weights is indeed ``natural" , in the sense that it can be interpreted as a version of Sawyer--testing conditions (one may compare it with \eqref{testcond}). This type of condition was first introduced by Sawyer \cite{ES} in 1982 in studying the behavior of Hardy-Littlewood maximal operators acting on weighted $L^p$ spaces, and later, the same idea has been applied by many authors to study other function spaces and operators, such as \cite{APR, PRB}.

\end{rem}

The following  is our main result in this subsection.

\begin{thm} Let $q>1$, $1<s<q'$, $\alpha>-1$, $u \in H(\R_+^2)$, $\varphi: \R_+^2  \mapsto \R_+^2$ be a holomorphic self-mapping and $\omega^{s'} \in \textbf{B}_{u, \varphi}^{\alpha, q}$, where $q'$ (respectively, $s'$) is the conjugate of $q$ (respectively, $s$). Let further, $W_{u, \varphi}: A_\alpha^q \mapsto A_\alpha^q$ be bounded. Then the following weighted estimate holds
\begin{equation} \label{20190930eq01}
\int_{\R_+^2} |u(z)|^q |f(\varphi(z))|^q \omega(z) dA_\alpha(z) \lesssim \left[\omega^{s'} \right]^{\frac{1}{s'}}_{\textbf{B}_{u, \varphi}^{\alpha, q}}\|f\|_{q, \alpha}^q,
\end{equation}
where the implicit constant in the above estimate is independent of the choice of $f$ and the weight $w$. In particular,  the measure  $\mu_{u, \varphi, p, w, \alpha}$ is a $1$-Carleson measure. Here
  $\mu_{u, \varphi, p, w, \alpha}$ is defined by
$$
\int_{\R^2_+} f d\mu_{u, \varphi, p, w, \alpha}=\int_{\R_+^2} |u(z)|^q |f(\varphi(z))|^q \omega(z) dA_\alpha(z), \quad f \ \textrm{is measurable}.
$$

\end{thm}

\begin{proof}
It suffices for us to prove the estimate \eqref{20190930eq01}, while the proof that the measure  $\mu_{u, \varphi, p, w, \alpha}$ is $1$-Carleson measure is standard and follows from a simply modification of its unit ball analog (see, e.g., \cite[Theorem 2.25]{Zhu}). Therefore, we omit the proof here.

The proof of the estimate \eqref{20190930eq01} follows from the spirit of Proposition \ref{20190914prop01} and Theorem \ref{190313cor01}. First, following the argument in the estimate \eqref{190212eq01}, we have
\begin{eqnarray*}
&&\int_{\R_+^2} |u(z)|^q |C_\varphi f(z)|^q \omega(z) dA_\alpha(z) \\
&\lesssim&  \int_{\R_+^2} |f(\zeta)| \left( \int_{\R_+^2} \frac{|f \circ \varphi(z)|^{q-1} |u(z)|^q w(z) dA_\alpha(z)}{|\bar{\zeta}-\varphi(z) |^{\alpha+2}} \right) A_\alpha(\zeta) \\
& = &\int_{\R_+^2} |f(\zeta)|  \left( \int_{\R^2_+} \frac{|f \circ \varphi(z)|^{q-1} |u(z)|^{\frac{q}{s}}}{|\bar{\zeta}-\varphi(z)|^{\frac{\alpha+2}{s}}} \cdot \frac{|u(z)|^{\frac{q}{s'}}\omega(z)}{|\bar{\zeta}-\varphi(z)|^{\frac{\alpha+2}{s'}}} dA_\alpha(z) \right) dA_\alpha(\zeta)\\
&\le& \int_{\R_+^2} |f(\zeta)|  \left( \int_{\R_+^2} \frac{|f \circ \varphi(z)|^{s(q-1)} |u(z)|^q}{|\bar{\zeta}-\varphi(z)|^{\alpha+2}} dA_\alpha(z) \right)^{\frac{1}{s}} \\
&& \quad \quad \quad \quad\quad \quad \quad \quad \cdot \left( \int_{\R_+^2} \frac{|u(z)|^q \omega^{s'}(z)}{|\bar{\zeta}-\varphi(z)|^{\alpha+2}} dA_\alpha(z) \right)^{\frac{1}{s'}} dA_\alpha(\zeta) \\
&\le& \left[\omega^{s'} \right]^{\frac{1}{s'}}_{\textbf{B}_{u, \varphi}^{\alpha, q}}\int_{\R_+^2} |f(\zeta)| \left( \int_{\R_+^2} \frac{|f(z)|^{s(q-1)}}{|\bar{\zeta}-z|^{\alpha+2}} d\mu_{u, \varphi, p, \alpha}(z) \right)^{\frac{1}{s}} dA_\alpha(\xi).
\end{eqnarray*}
Moreover, using Lemma \ref{sparse-1}, we have
\begin{eqnarray*}  \label{020202}
&& \int_{\R_+^2} |f(\zeta)| \left( \int_{\R_+^2} \frac{|f(z)|^{s(q-1)}}{|\bar{\zeta}-z|^{\alpha+2}} d\mu_{u, \varphi, p, \alpha}(z) \right)^{\frac{1}{s}} dA_\alpha(\xi) \\
&\lesssim & \int_{\R_+^2}  |f(\zeta)| \left( \sum_{i=1}^3 \sum_{I \in \calD^i} \frac{\one_{Q_I}(\zeta)}{A_\alpha(Q_I)} \int_{Q_I} |f(z)|^{s(q-1)}dA_\alpha(z) \right)^{\frac{1}{s}} dA_\alpha(\zeta) \\
&\lesssim& \int_{\R_+^2} |f(\zeta)|  \sum_{i=1}^3 \sum_{I \in D^i} \frac{\one_{Q_I}(\zeta)}{A_\alpha(Q_I)^{\frac{1}{s}}} \left( \int_{Q_I} |f(z)|^{s(q-1)} dA_\alpha(z) \right)^{\frac{1}{s}} dA_\alpha(z)  \\
& = &\sum_{i=1}^3 \sum_{I \in \calD^i} \frac{1}{A_\alpha(Q_I)^{\frac{1}{s}}} \left(\int_{Q_I} |f(z)| dA_\alpha(z) \right) \cdot \left( \int_{Q_I} |f(z)|^{s(q-1)} dA_\alpha(z) \right)^{\frac{1}{s}} \\
&= & \sum_{i=1}^3 \sum_{I \in \calD^i} A_\alpha(Q_I) \langle |f| \rangle_{Q_I} \langle |f|^{q-1} \rangle_{Q_I, s}  \\
& \lesssim& \sum_{i=1}^3 \sum_{I \in \calD^i} A_\alpha(Q^{\textrm{up}}_I) \langle |f| \rangle_{Q_I} \langle |f|^{q-1} \rangle_{Q_I, s} \\
&\lesssim& \int_{\R_+^2} \calM(|f|) \left(\calM( |f|^{(q-1)s}) \right)^{\frac{1}{s}} dA_\alpha(z) \\
&\le& \left( \int_{\R_+^2} \left| \calM(|f|) \right|^q dA_\alpha(z) \right)^{\frac{1}{q}} \cdot  \left(  \int_{\R_+^2} \left(\calM( |f|^{(q-1)s}) \right)^{\frac{q'}{s}} dA_\alpha(z) \right)^{\frac{1}{q'}} \\
&\lesssim& \|f\|_q \cdot \|f\|_q^{q-1}=\|f\|_q^q,
\end{eqnarray*}
where in the last inequality, we use the fact that $s<q'$ and the Hardy-Littlewood maximal operator is bounded on $L_\alpha^{\frac{q'}{s}}$. The desired result follows from combining these estimates.
\end{proof}

%--------------------------------------------------------------------------------
\section{Unit ball analog}

In this section, we extend our main results to the unit ball case, whose proof follows closely from the upper half plane case, and therefore, we would like to leave the details to the interested reader. Let us first recall some basic definitions.

Let $\B$ be the unit ball in $\C^n$ and $\SSS$ be its boundary. For $\alpha>-1$ and $p \ge 1$, the weighted Bergman space $A_\alpha^p(\B)$ is defined to be the space of holomorphic functions on $\B$ satisfying
$$
\|f\|_{A^p_\alpha(\B)}^p:=c_\alpha \int_{\B} |f(z)|^p (1-|z|^2)^\alpha dV(z)<\infty,
$$
where $dV(z)$ is the standard Lebesgue measure on $\B$ and $c_\alpha$ is chosen so that the measure $c_\alpha(1-|z|^2)^{\alpha}dV(z)$ is a probability measure on $\B$.

Next we recall some geometric facts on the unit ball $\B$. For $a \in \B$, let $\Phi_a$ be the involutive automorphism of $\B$ that interchanges $a$ and $0$, that is,
$$
\Phi_a \circ \Phi_a=\textrm{id}, \quad  \Phi_a(0)=a \quad \textrm{and} \quad \Phi_a(a)=0.
$$
This allows us to define the \emph{Bergman metric}  $\beta$ on $\B$, by
$$
\beta(z, w)=\frac{1}{2} \log \frac{1+|\Phi_z(w)|}{1-|\Phi_z(w)|}, \quad z, w \in \B.
$$
We are ready to construct the sparse collections on $\B$. To start with, we recall that, as in Definition \ref{dyadicgrid}, there exists a dyadic grid on the unit sphere $\SSS$, and we denote it as $\calD:=\{ Q_i^k\}_{i, k \in \Z}$ as usual (see, e.g., \cite{RTW}).

The following result can be understood as a version of the Whitney decomposition of the unit ball (in particular, this decomposition is parallel to the decomposition in Lemma \ref{sparse}).
\begin{prop}[\cite{RTW}] \label{dyadicgridball}
Let $\alpha>-1, \lambda, \theta>0$ and $\calD:=\{ Q_i^k\}_{i, k \in \Z}$ be a dyadic grid on $\SSS$ . Then there exists a collection of points in $\B$, which is denoted as $\calT$, satisfying the following properties:
\begin{enumerate}
\item [(1)] The set $\calT$ has a one-to-one correspondence with $\calD$. Moreover, we can write
$$
\calT=\bigcup_{N=1}^\infty \bigcup_{j=1}^{J_N} \left\{ c_j^N \right\},
$$
where $c_j^N \in \SSS_{\left(N+\frac{1}{2} \right) \theta}$, the sphere of radius $\left(N+\frac{1}{2} \right) \theta$ in the Bergman metric, $\frac{c_j^N}{|c_j^N|} \in Q_j^N$, and $J_N \ge 0$ is the number of cubes in $N$-th generation in $\calD$, which only depends on $N, n, \theta $ and $\lambda$;

\item [(2)]
$\calT$ has a tree structure, that is, we say $c_i^{N+1}$ is a child of $c_j^N$ for some $i, j>0$, if $ \frac{c_i^{N+1}}{|c_i^{N+1}|} \in  Q_j^N$.
\end{enumerate}
Moreover, for each $a \in \calT$, we can find a Borel set $K_a$, such that the following properties hold:
\begin{enumerate}
\item [(a)] $\B=\bigcup\limits_{a \in \calT} K_a$ and the sets $K_a$ are pairwise disjoint. Furthermore, there are constants $C_1$ and $C_2$ depending on $\lambda$ and $\theta$ such that for all $a \in \calT$ there holds:
$$
B_\beta(a, C_1) \subset K_\alpha \subset B_\beta(\alpha, C_2);
$$
\item [(b)] $V_\alpha(\widehat{K_a}) \simeq V_\alpha(K_a) \simeq (1-|a|^2)^{n+1+\alpha} \simeq e^{-2N\theta(n+1+\alpha)}$, for $a \in\B$ with $N\theta \le \beta(0, a)<(N+1)\theta$, where
$$
\widehat{K_\alpha}:=\bigcup_{a' \in \calT: a' \ \textrm{is a dyadic descendant of} \ a} K_{a'};
$$
\item [(c)] Every element of $\calT$ has at most $e^{2n\theta}$ children.
\end{enumerate}
\end{prop}

\begin{defn}
\begin{enumerate}
\item [(1)] We refer to $\calT$ as a \emph{Bergman tree} associated to the dyadic grid $\calD$, $a \in \calT$ the \emph{center} of $K_a$, $\B=\bigcup\limits_{a \in \calT} K_a$ the \emph{Whitney decomposition of $\B$} with respect to $\calD$ and
$\widehat{K_\alpha}$ the \emph{dyadic tent} under $K_\alpha$;
\item [(2)] For any cap $Q_j^N \in \calD$, we define its \emph{upper Whitney box} to be $K_{c_j^N}$, and its \emph{dyadic tent} to be $\widehat{K_{c_j^N}}$;
\item [(3)] For any $\calD$, a dyadic grid on $\SSS$, we define $Q_\calD:=\bigcup\limits_{Q_j^N \in \calD} \left\{ \widehat{K_{c_j^N}}\right\}=\bigcup\limits_{Q \in \calD} \left\{ \widehat{K_{c(Q)}}\right\}$ to be the collection of all its dyadic tents, where we use $c$ to denote the one-to-one correspondence between $\calD$ and $\calT$.
\end{enumerate}
\end{defn}

Note that an easy consequence of Proposition  \ref{dyadicgridball}, (b), is that for each dyadic grid $\calD$ on $\SSS$, the collection $Q_\calD$ is a sparse collection.

The following proposition can be understood as the replacement for a collection of dyadic grids that can approximate any cube appropriately in the setting of $\B$.

\begin{prop} [{\cite[Lemma 3]{RTW}}] \label{Meiball}
There is a finite collection of Bergman trees $\{\calT^\ell\}_{\ell=1}^M$ such that for all $z \in \B$, there is a tree $\calT$ from the finite collection and an $a \in \calT$ such that the dyadic tent $\widehat{K_\alpha}$ contains the tent $T_z$ and $V_\alpha(\widehat{K_a}) \simeq V_\alpha(T_z)$, where $T_z:=\left\{w \in \B: \left|1-\bar{w} \frac{z}{|z|} \right|<1-|z| \right\}$ is the Carleson tent over $z$.
\end{prop}

\begin{rem}
The above proposition allows us to extend the definition of upper Whitney box in the following way: let $E \subset \SSS$ (here, $E$ is not necessarily a cap in any dyadic grid $\calD$ on $\SSS$) with satisfying $E=\SSS \cap \partial T_z$ for some $z \in \B$. Then by Proposition \ref{Meiball}, we can find some $a=a_z \in \B$, such that $T_z \subseteq \widehat{K_\alpha}$ and $V_\alpha(\widehat{K_a}) \simeq V_\alpha(T_z)$. In this case, we define the \emph{upper Whitney box} of $E$ as $K_a$.
\end{rem}

Therefore, combing Propositions \ref{dyadicgridball} and  \ref{Meiball}, we see that there exists dyadic grids $\calD^1, \dots, \calD^M$ on $\SSS$, which corresponds to $\calT^1, \dots, \calT^M$, respectively, playing the same roles as $\calD^1, \dots, \calD^3$ defined in \eqref{explicitconst} for the upper half plane case.

\medskip

We are ready to state our main results for the boundedness and compactness of the weighted composition operators acting on the weighted Bergman spaces on $\B$.

\begin{thm}
 Let $q \ge 1, \alpha>-1$, $u\in H(\B)$ and $\varphi: \B \rightarrow \B$ be a holomorphic mapping. Then the following statements are equivalent.
\begin{enumerate}
\item [(i)] $W_{u, \varphi}: A^q_\alpha(\B)  \mapsto A^q_\alpha(\B)$ is bounded;
\item [(ii)] For any $f \in A^q_\alpha(\B)$,
$$
\|W_{u, \varphi} f\|_{A^q_\alpha(\B)}^q \lesssim \inf\limits_{N \in \N,  1 \le N \le q} \left( \sum_{i=1}^M \sum_{Q \in \calD^i} V_\alpha\left(\widehat{K_{c(Q)}}\right) \left\langle |f|^N \right\rangle_{\widehat{K_{c(Q)}}} \cdot  \left\langle |f|^{q-N} \right\rangle_{\widehat{K_{c(Q)}}} \right).
$$
\end{enumerate}
\end{thm}

\begin{thm}
 Let $2p>q > p \ge 1, \alpha>-1$, $u\in H(\B)$ and $\varphi: \B \rightarrow \B$ be a holomorphic mapping. Suppose
$$
\Z_{p, q}:=\left\{ N \in \N:  ~~~~ N \ge 1, N<p<q<p+N \right\} \neq \emptyset.
$$
Then the following statements are equivalent:
\begin{enumerate}
\item [(i)] $W_{u, \varphi}: A^p_\alpha(\B) \mapsto A^q_\alpha(\B)$ is bounded;
\item [(ii)] For any $f \in A^p_\alpha(\B)$,
$$
\|W_{u, \varphi} f\|_{A^q_\alpha(\B)}^q \lesssim \inf\limits_{N \in \Z_{p, q}} \left( \sum_{i=1}^M \sum_{Q \in \calD^i} V^{\frac{q}{p}}_\alpha\left(\widehat{K_{c(Q)}}\right) \left\langle |f|^N \right\rangle_{\widehat{K_{c(Q)}}} \cdot  \left\langle |f|^{q-N} \right\rangle_{\widehat{K_{c(Q)}}} \right).
$$
\end{enumerate}
\end{thm}

\begin{thm}
 Let $q \ge 1, \alpha>-1$, $u\in H(\B)$ and $\varphi: \B \rightarrow \B$ be a holomorphic mapping.  If $W_{u, \varphi}: A^q_\alpha(\B) \mapsto A^q_\alpha(\B)$ is bounded,  then the following statements are equivalent.
\begin{enumerate}
\item [ (i)] $W_{u, \varphi}: A^q_\alpha(\B) \mapsto A^q_\alpha(\B)$ is compact;
\item [ (ii)]  Let $1 \le N \le q, N \in \N$ and $1<\gamma<\frac{q}{q-N}$, or $q=N \in \N$ and $\gamma>1$. Let further, $\{K_j\}_{j \ge 1}$ be a sequence of exhausting sets of $\B$, that is, $\{K_j\}_{j \ge 0}$ is a collection of compact sets in $\B$, satisfying $K_1 \subsetneq K_2 \subsetneq \dots K_j \subsetneq \dots \subsetneq \B$, and $\bigcup\limits_{j=1}^\infty K_j=\B$. Then for any bounded set $\{f_m\}_{m \ge 1} \subset A^p_\alpha(\B)$ with $f_m\rightarrow 0$ as $m\rightarrow\infty$, uniformly on compact subsets of $\B$,
\begin{eqnarray*}
&&\lim_{j \to \infty}  \sup_{m \ge 1} \bigg(\sum_{i=1}^M \sum_{Q \in \calD^i, K_{c(Q)} \cap K_j=\emptyset} \mu_{u, \varphi, q, \alpha} \left(\widehat{K_{c(Q)}}\right)^{\frac{1}{\gamma'}} \\
&& \quad  \quad  \quad \quad  \quad \quad \quad  \quad  \quad \cdot V^{\frac{1}{\gamma}}_\alpha\left(\widehat{K_{c(Q)}} \right) \langle |f_m|^N \rangle_{\widehat{K_{c(Q)}}} \langle |f_m|^{q-N} \rangle_{\widehat{K_{c(Q)}}, \gamma} \bigg)=0.
\end{eqnarray*}
\end{enumerate}
\end{thm}

Next, we extend the weighted estimates \eqref{20190930eq01} to the unit ball. We first need the following $\mathcal{B}_{u, \varphi}^{\alpha, q}$ weights, which is a unit ball analog of the $\textbf{B}_{u, \varphi}^{\alpha, q}$ weights.

\begin{defn}
Given $\alpha>-1$, $q>1$, $u \in H(\B)$ and $\varphi: \B \to \B$, the weight class $\mathcal{B}_{u, \varphi}^{\alpha, q}$ is defined to be the collection of all weights $\omega$ on $\B$ satisfying
\begin{equation} \label{Bpcondition}
\left[\omega \right]_{\mathcal{B}_{u, \varphi}^{\alpha, q}}:=\sup_{\zeta \in \B} \int_{\B} \frac{|u(z)|^q \omega(z)}{\left| 1-\langle \zeta, \varphi(z)\rangle \right|^{\alpha+n+1}} dV_\alpha(z)<\infty. \nonumber
\end{equation}
\end{defn}

\begin{thm} Let $q>1$, $1<s<q'$, $\alpha>-1$, $u \in H(\B)$, $\varphi: \B  \mapsto \B$ be a holomorphic self-mapping and $\omega^{s'} \in \mathcal{B}_{u, \varphi}^{\alpha, q}$, where $q'$ (respectively, $s'$) is the conjugate of $q$ (respectively, $s$). Let further, $W_{u, \varphi}: A_\alpha^q(\B) \mapsto A_\alpha^q(\B)$ be bounded. Then the following weighted estimate holds.
\begin{equation} \label{20190930eq02}
\int_{\B} |u(z)|^q |f(\varphi(z))|^q \omega(z) dV_\alpha(z) \lesssim \left[\omega^{s'} \right]^{\frac{1}{s'}}_{\mathcal{B}_{u, \varphi}^{\alpha, q}}\|f\|_{A^q_\alpha(\B)}^q,
\end{equation}
where the implicit constant in the above estimate is independent of the choice of $f$ and the weight $w$. In particular,  the measure  $\mu_{u, \varphi, p, w, \alpha}$ is a $1$-Carleson measure. Here
  $\mu_{u, \varphi, p, w, \alpha}$ is defined by
$$
\int_{\B} f d\mu_{u, \varphi, p, w, \alpha}=\int_{\B} |u(z)|^q |f(\varphi(z))|^q \omega(z) dV_\alpha(z), \quad f \ \textrm{is measurable}.
$$
\end{thm}

%----------------------------------------------------------------------------------
\section{Further remarks}
We conclude the article with several remarks. Our main results, and the proofs, are a model case for a wider range of results in studying complex function theory and weighted estimates via sparse domination. Some  possible extensions to the main results of this paper are as follows.

\begin{enumerate}
\item [$(a)$] Establish the results for more general domains. That is, find the sparse bounds for weighted composition operators acting between weighted Bergman spaces on polydics, Hartogs domains, Thullen domain and etc.

\smallskip

\item [$(b)$] Study the sparse bounds and corresponding weighted estimates of  weighted composition operators acting on some M\"obius invariant function spaces. Typical examples of these spaces include the Bloch space  $\calB$, $\mathcal{Q}_p$ and  $\mathcal{Q}_K$ spaces (see, e.g., \cite{WZ}).

\smallskip

\item [$(c)$] Introduce more general weighted estimates, for example, weighted estimates with matrix weights. This would  encounter extra difficulties, for example, one needs to figure out a correct notion of convex body domination in the setting of complex function spaces.
\end{enumerate}

\textbf{Acknowledgements.} The second author was partially supported by the  Science and Technology Development
Fund, Macau SAR (File no. 186/2017/A3).   BDW's research supported in part by a National Science Foundation DMS grants \# 1560955 and \# 1800057 and Australian
 Research Council -- DP 190100970.

\end{document}